\documentclass[12pt]{amsart}

\usepackage{tikz}
\usetikzlibrary{snakes}
\usepackage[colorlinks=true, linkcolor=blue, anchorcolor=blue, citecolor=blue, filecolor=blue, menucolor= blue, urlcolor=blue]{hyperref}
\usepackage{amsmath,amssymb,amsthm,amscd}
\usepackage{verbatim}
\usepackage{comment}
\usepackage{multirow}
\usepackage{mathtools}
\usepackage{enumitem}
\usepackage{pgf,tikz}
\usetikzlibrary{arrows}
\usepackage[normalem]{ulem}
\usepackage{graphicx}

\usepackage[margin=1in]{geometry} 

\numberwithin{equation}{section}

\newtheorem{theorem}{Theorem}[section]
\newtheorem{proposition}[theorem]{Proposition}

\newtheorem{corollary}[theorem]{Corollary}

\newtheorem*{theorem*}{Theorem}

\theoremstyle{definition}

\newtheorem{notation}[theorem]{Notation}
\newtheorem{example}[theorem]{Example}
\newtheorem{remark}[theorem]{Remark}

\definecolor{MyDarkGreen}{cmyk}{0.7,0,1,0}

\def\cocoa{{\hbox{\rm C\kern-.13em o\kern-.07em C\kern-.13em o\kern-.15em A}}}

\begin{document}

\title[The singular locus of a hyperplane arrangement in $\mathbb P^n$]{Schemes supported on the singular locus of a hyperplane arrangement in $\mathbb P^n$}

\author{J.\ Migliore} 
\address{Department of Mathematics \\
University of Notre Dame \\
Notre Dame, IN 46556 USA}
 \email{migliore.1@nd.edu}
 
 \author{U. Nagel} 
\address{Department of Mathematics \\
University of Kentucky \\
715 Patterson Office Tower, \\
Lexington, KY 40506-0027 USA}
 \email{uwe.nagel@uky.edu}
 
 \author{H. Schenck} 
\address{Department of Mathematics \\
Iowa State University \\
Ames, Iowa 50011-2064 USA}
 \email{hschenck@iastate.edu}

\begin{abstract} 
We introduce the use of liaison addition to the study of hyperplane arrangements. 
For an arrangement, $\mathcal A$, of hyperplanes in $\mathbb P^n$,  $\mathcal A$ is free if $R/J$ is Cohen-Macaulay, where $J$ is the Jacobian ideal of $\mathcal A$. Terao's conjecture says that freeness of $\mathcal A$ is determined by the combinatorics of the intersection lattice of $\mathcal A$. We study the Cohen-Macaulayness of three other ideals, all unmixed, that are closely related to $\mathcal A$. 
Let $\overline J = \mathfrak q_1 \cap \dots \cap \mathfrak q_s$ be the intersection of height two primary components of $J$ and $\sqrt{J} = \mathfrak p_1 \cap \dots \cap \mathfrak p_s$ be the radical of $J$.
Our third ideal is $\mathfrak p_1^{b_1} \cap \dots \cap \mathfrak p_s^{b_s}$ for suitable $b_1,\dots, b_s$. With a fairly mild hypothesis we use liaison addition to show that all three of these ideals are Cohen-Macaulay.  When our hypothesis does not hold,  we show that these ideals are not necessarily Cohen-Macaulay, and further that Cohen-Macaulayness of any of these ideals does not imply Cohen-Macaulayness of any of the others.  While we do not study the freeness of $\mathcal A$, we give an example to show that the Betti diagrams can vary even for  arrangements with the same combinatorics.

In the second part of the paper we study the situation when the hypothesis does not hold. For equidimensional curves in $\mathbb P^3$, the Hartshorne-Rao module from liaison theory measures the failure of an ideal to be Cohen-Macaulay, degree by degree, and also determines the even liaison class of such a curve. We show that for any positive integer $r$ there is an arrangement $\mathcal A$ for which $R/\overline J$ fails to be Cohen-Macaulay in only one degree, and this failure is by $r$;  we also give an analogous result for $\sqrt{J}$. We draw consequences for the corresponding even liaison class of the curve defined by $\overline J$ or by $\sqrt{J}$.   
\end{abstract}

\date{\today}

\thanks{
{\bf Acknowledgements}: 
Migliore was partially supported by Simons Foundation grant \#309556.
Nagel was partially supported by Simons Foundation grant \#317096. 
Schenck was partially supported by National Science Foundation grant \#1818646.
}



\maketitle



\section{Introduction}

Let $R = k[x_0,\dots,x_n]$, where $k$ is a field of characteristic zero.  Let $F$ be a homogeneous polynomial and let $J$ be the Jacobian ideal of $F$. Of course a huge number of papers have studied Jacobian ideals and the schemes that they define. 
In this paper we focus on homogeneous polynomials $F$ defining hyperplane arrangements. Specifically, let $\mathcal A$ be a hyperplane arrangement in $\mathbb P^n$ containing $d$ hyperplanes. Let $F$ be a product of linear forms defining $\mathcal A$ and let $J$ be the Jacobian ideal of $F$. 

The  ideal $J$ has height two. Let $E$ be the  syzygy module of $J$, so we have an exact sequence
\[
0 \rightarrow E \rightarrow \bigoplus_{i=1}^{n+1} R(1-d) \rightarrow J \rightarrow 0,
\]
and let 
\[
J = \mathfrak q_1 \cap \dots \cap \mathfrak q_r
\]
be a primary decomposition of $J$. Let $\mathfrak p_i$ be the associated prime of $\mathfrak q_i$.
Many papers have studied the situation when $E$ is a free module, i.e. when $R/J$ is a Cohen-Macaulay ring.
The papers \cite{schenck} and \cite{MS} studied issues of when the sheafification, $\tilde{E}$, is locally free, and \cite{DS} studied  the saturation, $J^{sat}$, of $J$ with respect to the homogeneous maximal ideal $\mathfrak m$ of $R$, (but in a more general setting, with applications to line arrangements, in the case $n=2$). Notice that the saturation of $J$ means that we remove a primary component whose associated prime is $\mathfrak m$ (if it exists), but that when $n \geq 3$, $J^{sat}$ could still have embedded (but not isolated) components of height $> 2$.  That is, even $J^{sat}$ may fail to be unmixed.

In this paper we take this one step further and remove the remaining components from the primary decomposition that have height $> 2$.   We will be interested in three {\em unmixed} ideals that arise very naturally from $J$. 
First, we consider the intersection of the codimension two ideals in a primary decomposition of $J$, which we will denote $\overline J$. The second ideal that we study is the radical of $J$, denoted $\sqrt{J}$, which is just the intersection of the associated primes of $J$. While these first two ideals are most important for us, from $\sqrt{J} = \bigcap \mathfrak p_i$ we can also replace each height two associated prime $\mathfrak p_i$ with a suitable power $\mathfrak p_i^{b_i}$.

From a geometric point of view, the first two of these ideals define schemes that can be interpreted as representing the singularity of $\mathcal A$, once embedded components have been removed. The ideal $\overline J$ is, in general, non-reduced but its components (taken individually) are complete intersections of height two supported on linear varieties. The radical ideal $\sqrt{J}$ defines the union of linear varieties on which the singularity is supported. The third ideal is less natural, but it is interesting to ask to what extent we can ``fatten up" the components of $\sqrt{J}$ and still get results about Cohen-Macaulayness. In general, the latter scheme is not a local complete intersection.  Our goal is to show that ``most of the time" all three of  these schemes are arithmetically Cohen-Macaulay (ACM), i.e.  $R/\overline{J}$, $R / \sqrt{J}$ and $R/\bigcap \mathfrak p_i^{b_i}$ (with the permitted range of exponents to be described later) are Cohen-Macaulay rings. For the first two we obtain the following  result (Theorem \ref{mainthm}, Corollary~\ref{sing is acm}, Corollary \ref{Pn}):

\medskip

\noindent {\bf Theorem.} {\em Let $\mathcal A$ be a hyperplane arrangement in $\mathbb P^n$ defined by a product, $F$ of linear forms. Let $J$, $\sqrt{J}$ and $\overline J$ be the ideals defined above. 
Assume that no linear factor of $F$  is in the associated prime for any two \underline{non-reduced}  components of $\overline J$. Then both $R/{\overline J}$ and $R/\sqrt{J}$ are Cohen-Macaulay.}

\medskip

\noindent The third ideal, $\bigcap \mathfrak p_i^{a_i}$, will be described in Corollary \ref{fat is acm} and the paragraph preceding it, and our result will be that with the same assumption, this algebra is also Cohen-Macaulay. This will not play as big a role in this paper as the other two ideals.

As an application of the above theorem, in \S \ref{graphic} we consider graphic arrangements, which are certain arrangements arising from graphs. We translate the result mentioned above as follows (see Corollary \ref{graphic result}). Let $G$ be a graph and assume that no two 3-cycles of $G$ share an edge. Let $\mathcal A_G$ be the arrangement derived from $G$ (see \S \ref{graphic}). Then $R/\sqrt{J}$ and $R/\overline{J}$ are Cohen-Macaulay. 

Let us give a geometric explanation of the hypothesis in the above theorem, in the special case of $\mathbb P^3$. The ideal $\overline J$ defines a scheme $\overline V$ supported on a union of lines. Some of the components of $\overline V$ may not be reduced. We allow two or more such non-reduced components to meet, but if they do meet we assume that the plane spanned by the supports is not a plane in the arrangement. Other than that, we do not care how reduced components meet each other or how they meet non-reduced components.

For most of this paper we will focus on the case $n=3$, and  we will obtain the above result by taking general hyperplane sections. 

While the focus in the above theorem is on whether or not a ring is Cohen-Macaulay, in the case of a curve $C$  in $\mathbb P^3$ there is actually a very useful graded module that measures the {\em failure} of $R/I_C$ to be Cohen-Macaulay (i.e. the failure of $C$ to be {\em arithmetically Cohen-Macaulay (ACM)}) degree by degree. This is the so-called Hartshorne-Rao module of $C$, $M(C)$, which we recall in the next section. 
Recall that the Hartshorne-Rao module plays a central role in the even liaison class of $C$. Thus we also study the even liaison class of the schemes $\overline C_F$ and $C_F^{red}$. Proposition \ref{liaison addition for products} and Proposition \ref{BDL for arrangements} show how liaison addition and basic double linkage (two tools from liaison theory) can be applied to arrangements. 

To state the second main result, we refine our  notation so that if $\mathcal A$ is a hyperplane arrangement defined by a product, $F$, of linear forms, then the scheme defined by the top dimensional part of the corresponding Jacobian ideal is denoted by $\overline C_F$ and the scheme defined by the radical $\sqrt{J}$ is $C_F^{red}$. 
The main result of \S \ref{liaison sect} is the following (see Theorem \ref{buchs thm} and Corollary \ref{main thm for rad}), where we show that both curves can be made to fail to be ACM in only one degree, and nevertheless the failure can be made as large as desired! 

\medskip

\noindent {\bf Theorem.} 
{\em 
Let $r \geq 1$ be a positive integer. Then:

\begin{itemize}

\item[\em (i)] There exists a positive integer $N$ and a product of linear forms $F$, defining an arrangement $\mathcal A_F$ in $\mathbb P^3$, such that 
\[
\dim M(\overline C_F)_{N} = r
\]
and all other components of $M(\overline C_F)$ are zero. 

\item[\em (ii)] There exists a positive integer $N'$ and a product of linear forms $F'$, defining an arrangement $\mathcal A_{F'}$ in $\mathbb P^3$, such that 
\[
\dim M( C_{F'}^{red})_{N'} = r
\]
and all other components of $M(C_{F'}^{red})$ are zero. 

\item[\em (iii)] For each $h \geq 1$ we can replace $N$ by $N+h$ and find a polynomial $G$ so that 
\[
\dim M(\overline C_G)_{N+h} = r, 
\]
and all other components of $M(\overline C_G)$ are zero. The curve $\overline C_G$ is in the same even liaison class as $\overline C_F$. The analogous result for $C^{red}$ also holds. 

\end{itemize}
}

As an application of the third part of the above theorem, we note that once an even liaison class has a curve of either type arising from an arrangement, the same class will have infinitely many such curves (see Corollary \ref{buchs even liaison} and Corollary \ref{main thm for rad}).

The idea introduced in this paper and used for both main results is to build up our curves using liaison addition and basic double linkage. Of course these tools have existed for a long time, but the idea of using them in the context of the singular scheme or locus of a hyperplane arrangement is mostly new. (The use of basic double linkage was begun in \cite{GHM3} in the study of star configurations, which in codimension two are a very special case of singular loci of hyperplane arrangements,  but our application here is more general, and the use of liaison addition is entirely new.)

 One important starting point for this paper, particularly the second theorem above, is an example due to Musta\c t\v a and the third author \cite{MS}  which shows that not all planar arrangements in $\mathbb P^3$ give rise to ideals $\overline J$ for which $R/\overline J$ is Cohen-Macaulay. Modifications of this example  are used in our work.


Freeness of $\mathcal A$ is equivalent to $R/J$ being Cohen-Macaulay, and Terao's conjecture  (cf. \cite{OT} Conjecture 4.138)  that freeness of $\mathcal A$ is determined by the combinatorics of the intersection lattice has motivated much work in the field. In this paper we study several natural questions related to the freeness of $\mathcal A$ (but not the freeness question itself).

\begin{itemize}


\item {\em If $\mathcal A$ is free, must $R/\sqrt{J}$ be Cohen-Macaulay?} No (Example \ref{free not cm} gives a counterexample).

\item {\em Does the Cohen-Macaulayness of either $R/\overline J$ or $R/\sqrt{J}$ imply Cohen-Macaulayness of the other?} No (Example \ref{ex poss} shows all combinations can occur).

\item {\em Does the failure of Cohen-Macaulayness for $R/\overline J$ and $R/\sqrt{J}$ imply the failure of Cohen-Macaulayness for $R/\bigcap \mathfrak p_i^{b_i}$?}  No (Example \ref{fat is CM} gives an ideal for which both $R/\sqrt{J}$ and $R/\overline J$ fail to be Cohen-Macaulay, but $R/\bigcap \mathfrak p_i^{b_i}$ is Cohen-Macaulay. In fact, in this example even the symbolic square of $\sqrt{J}$, i.e. the case where $b_i = 2$ for all $i$, is Cohen-Macaulay).

\item {\em Are there other Hartshorne-Rao modules that arise for either $\overline C$ or $C^{red}$ in addition to the ones described in the second theorem above?} Yes (Example \ref{not all classes} shows that both can happen).

\end{itemize}

Our first theorem above shows that under a certain fairly weak condition, both $R/\overline J$ and $R/\sqrt{J}$ must be Cohen-Macaulay, regardless of whether $R/J$ is Cohen-Macaulay or not. In this sense, our results  are independent of the freeness of the arrangement. 
On the other hand,  while we do not address Terao's Conjecture  itself,  we do show that the Betti diagrams of these algebras can vary for arrangements with the same combinatorics (Example \ref{same comb}).

We conclude the paper in \S \ref{open q} with some open questions arising from our work.


\section{Tools} \label{tool section}

Let $R = k[x_0,\dots,x_n]$, where $k$ is a field of characteristic zero.  Recall that a subscheme $V \subset \mathbb P^n$ is {\em arithmetically Cohen-Macaulay} (denoted ACM) if the coordinate ring $R/I_V$ is a Cohen-Macaulay ring. For a general background to the tools used in this paper see \cite{mig-book}.

\begin{notation}  \label{main notation}
Let $\mathcal A$ be a hyperplane arrangement in $\mathbb P^n$. Let $F$ be the product of linear forms defining $\mathcal A$. Let $J = \langle F_{x_0}, \dots, F_{x_n} \rangle$ be the Jacobian ideal of $F$. Note that $J$ is not necessarily saturated, and its saturation is not necessarily unmixed. We make the following notation:

\medskip

\begin{itemize}
\item $J^{sat}$ is the saturation of $J$ with respect to the maximal ideal of $R$. 

\item $\overline{J}$ is the top-dimensional component of $J$. That is, $\overline{J}$ is the intersection of the height two primary ideals in a primary decomposition of $J$.

\item $V$ is the scheme defined by $J$ (i.e. $I_V = J^{sat}$). Note that $V$ may have embedded components, and $V$ is often non-reduced.

\item $\overline V$ is the scheme defined by $\overline J$. Note that $\overline V$ is equidimensional, but it is not necessarily reduced. We refer to $\overline V$ as the {\em top-dimensional part} of $V$.

\item $\sqrt{J}$ is the radical of $J$. Note that $\sqrt{J}$ is the intersection of the height two associated primes of $J$.

\item $V^{red}$ is the scheme defined by $\sqrt{J}$. This is the reduced scheme on which $\overline V$ is supported, and is the set-theoretic singular locus of $\mathcal A$. Note that $V^{red}$ is reduced and equidimensional, supported on a union of codimension two linear varieties.
\end{itemize}
\end{notation}

\begin{notation} \label{FG notation}
If $F$ and $G$ are both products of pairwise independent linear forms with no factors in common then $F$, $G$ and $FG$ all define hyperplane arrangements. In this setting we will denote these arrangements by $\mathcal A_F, \mathcal A_G, \mathcal A_{FG}$ respectively. We will denote the corresponding items in Notation \ref{main notation} with the appropriate subscripts; for instance, $V_{FG}^{red}$ is the scheme defined by the radical ideal $\sqrt{J_{FG}}$, where $J_{FG}$ is the Jacobian ideal of $FG$.
\end{notation}

\begin{notation}
Let $C \subset \mathbb P^3$ be an equidimensional curve with no embedded points. Then we denote by $M(C)$ the {\em Hartshorne-Rao module} of $C$, namely the graded module 
\[
M(C) = \bigoplus_{t \in \mathbb Z} H^1(\mathbb P^3, \mathcal I_C(t)). 
\]
Under our assumptions on $C$, $M(C)$ is a graded module of finite length. It is zero if and only if $C$ is ACM.
\end{notation}

First we recall the construction of liaison addition. This was introduced by P. Schwartau in his Ph.D.\ thesis \cite{Sw}, which was never published. However, a generalization was found and published in \cite{GM4}. The version that we need (and state) is entirely due to Schwartau (except that he gave it only for curves in $\mathbb P^3$), but we cite \cite{GM4} and \cite{mig-book} as the only available sources at this point.

\begin{theorem}[\cite{GM4} Corollary 1.6, Theorem 1.3 and Corollary 1.5, \cite{mig-book} section 3.2] \label{LA}
Let $V_1$ and $V_2$ be locally Cohen-Macaulay, equidimensional codimension two subschemes in $\mathbb P^n$.  Choose polynomials $F_1, F_2$ with $d_i = \deg F_i$ so that 
\[
F_1\in I_{V_1} \ \ \ \hbox{ and } \ \ \ F_2 \in I_{V_2}, 
\]
and furthermore $(F_1,F_2)$ is a regular sequence. Let $V$ be the complete intersection scheme defined by $(F_1,F_2)$. Let $I = F_2 I_{V_1} +  F_1 I_{V_2}$ and let $Z$ be the scheme defined by $I$. Then

\begin{itemize}
\item[(i)] If $V_1,V_2$ and $V$ pairwise have no common components then $Z = V_1 \cup V_2 \cup V$ as schemes.

\item[(ii)] $I$ is a saturated ideal.

\item[(iii)] If $h_X(t)$ denotes the Hilbert function of a scheme $X$ then we have
\[
h_Z(t) = h_V(t) + h_{V_1}(t-d_1) +  h_{V_2}(t-d_2).
\]

\item[(iv)] $Z$ is ACM if and only if both $V_1$ and $V_2$ are  ACM. More generally,  we have
\[
\bigoplus_{t \in \mathbb Z} H^i( \mathbb P^n, \mathcal I_Z (t)) \cong \bigoplus_{t \in \mathbb Z} H^i (\mathbb P^n, \mathcal I_{V_1}(t))(-d_2) \oplus \bigoplus_{t \in \mathbb Z} H^i (\mathbb P^n, \mathcal I_{V_2}(t))(-d_1)
\]
as graded $R$-modules, for $1 \leq i \leq n-2$. If the schemes are curves in $\mathbb P^3$, these cohomology modules are simply the Hartshorne-Rao modules.

\end{itemize}
\end{theorem}

%
%
%
%
%
%
%
%

The following construction, basic double linkage,  is obtained from liaison addition. We write down only the version that we need, but a much more general version exists, the so-called {\em basic double G-linkage} -- cf. for instance \cite{MN3} Lemma 3.4 and \cite{mig-book} Theorem~3.2.3 and Remark~3.2.4. Our version follows from liaison addition 
by taking $V_2$ to be the empty set and $I_{V_2} = R$.

\begin{proposition} \label{BDL}
Let $V_1$ be a locally Cohen-Macaulay, equidimensional  codimension two subscheme in $\mathbb P^n$. Choose polynomials $F_1, F_2$ with $d_i = \deg F_i$ and $F_1 \in I_{V_1}$ and $F_2 \in R$, such that $(F_1,F_2)$ is a regular sequence. Let $V$ be the complete intersection scheme defined by $(F_1,F_2)$. Let $I = F_2 I_{V_1} + (F_1)$ and let $Z$ be the scheme defined by $I$. Then

\begin{itemize}

\item[(i)] If $V_1$ and $V$ have no common components then $Z = V_1 \cup V$.

\item[(ii)] $I$ is a saturated ideal.

\item[(iii)] $Z$ is ACM if and only if $V_1$ is ACM. More generally, 
\[
\bigoplus_{t \in \mathbb Z} H^i( \mathbb P^n, \mathcal I_Z (t)) \cong \bigoplus_{t \in \mathbb Z} H^i (\mathbb P^n, \mathcal I_{V_1}(t))(-d_2) 
\]
for $1 \leq i \leq n-2$.

\item[(iv)] $Z$ is linked in two steps to $V_1$.

\end{itemize}

\end{proposition}

Although our first main result is for arrangements in $\mathbb P^n$, we only need the versions of  Theorem \ref{LA} and Proposition \ref{BDL} in $\mathbb P^3$. Following Notation~\ref{FG notation}, we will use the following  adaptation.  First we apply liaison addition to show how to combine arrangements (subject to a reasonable hypothesis) in order to obtain new curves with bigger Hartshorne-Rao module.

\begin{proposition} \label{liaison addition for products}
Let $F = L_1\dots L_m$ and $G = M_1 \dots M_p$ be products of pairwise independent linear forms in $R = k[x_0,x_1,x_2,x_3]$ with $m \geq 2$ and $p \geq 2$ such that $G$ does not vanish on any component of $C_F^{red}$ and $F$ does not vanish on any component of $C_G^{red}$. Then
\[
\begin{array}{c}
M(\overline C_{FG}) \cong M(\overline C_F)(-p) \oplus M(\overline C_G)(-m) \\
\hbox{and} \\
M(C_{FG}^{red}) \cong M(C_F^{red})(-p) \oplus M(C_G^{red})(-m).
\end{array}
\]
\end{proposition}

\begin{proof}
Since $F \in I_{\overline C_F}$ and $G \in I_{\overline C_G}$, this follows from Theorem \ref{LA}.
\end{proof}

This result gives us information, in some cases, about the union of two arrangements.

\begin{corollary} \label{LA for products cor}
With the assumptions of Proposition \ref{liaison addition for products}, if $\overline C_F$ and $\overline C_G$ (respectively $C_F^{red}$ and $C_G^{red}$) are both ACM then $\overline C_{FG}$ (respectively $C_{FG}^{red}$) is ACM. 
\end{corollary}

\begin{remark}
Notice that Corollary \ref{LA for products cor} does not show that if $\mathcal A_F$ and $\mathcal A_G$ are free then $\mathcal A_{FG}$ is free. Indeed, a simple example is when $F$ and $G$ both consist of products of three general linear forms. Then $\mathcal A_F$ and $\mathcal A_G$ are free, but $\mathcal A_{FG}$ is not. The only conclusion from Corollary \ref{LA for products cor} is that $\overline C_{FG}$ and $C_{FG}^{red}$ are both ACM (and in fact they are equal and form a star configuration, so this is already known \cite{GHM3}).
\end{remark}

Next we translate the notion of basic double linkage  to our work on arrangements. The notion of basic double linkage for curves in $\mathbb P^3$ was introduced by Lazarsfeld and Rao \cite{LR}.

\begin{proposition} \label{BDL for arrangements}
Let $F$ be a product of linear forms defining a hyperplane arrangement in $\mathbb P^3$.  Let $L$ be a linear form that does not vanish on any component of $C_F^{red}$. Then 

\begin{itemize}

\item[(i)] $M(\overline C_{LF}) \cong M(\overline C_F)(-1)$ and $M(C_{LF}^{red} ) \cong M(C_F^{red})(-1)$. 

\item[(ii)] For the saturated homogeneous ideals we have
\[
I_{\overline C_{LF}} = L \cdot I_{\overline C_F} + (F) \hbox{ and }
I_{C_{LF}}^{red} = L \cdot I_{C_F^{red}} + (F).
\]

\item[(iii)] In both cases the ACM property (if present) is preserved. More generally, the even liaison class of $\overline C$ is preserved and the even liaison class of $C^{red}$ is preserved   under basic double linkage.

\end{itemize}
\end{proposition}

\begin{proof}
By basic double linkage, we know that the ACM property is preserved and that the stated ideals are saturated. We only note that the components coming from factors of $F$ are already accounted for in $\overline C_F$ (resp. $C_F^{red}$), and because of the assumption on $L$ the new components are reduced lines coming from the intersection of $L$ and $F$. Thus this curve arises from the arrangement $\mathcal A_{LF}$.
\end{proof}

\begin{corollary} \label{loc free}
Let $F = L_1\dots L_m$ and $G = M_1 \dots M_p$ be products of pairwise independent linear forms in $R = k[x_0,x_1,x_2,x_3]$ with $m \geq 2$ and $p \geq 2$, such that $G$ does not vanish on any component of $C_F^{red}$ and $F$ does not vanish on any component of $C_G^{red}$. Let $L$ be a linear form that does not vanish on any component of $C_F^{red}$. Let $J_F$, $J_G$, $J_{FG}$ and $J_{LF}$ be the Jacobian ideals of $F$, of $G$, of $FG$ and of $LF$, respectively. 

\begin{itemize}

\item[(i)]  If the sheafification of the syzygy module of both $J_F$ and $J_G$ are locally free then so is the sheafification of the syzygy module of $J_{FG}$.

\item[(ii)] If the sheafification of the syzygy module of $J_F$ is locally free then so is the sheafification of the syzygy module of $J_{LF}$.

\end{itemize}
\end{corollary}

\begin{proof}
See \cite{BM2} Proposition 2.7.
\end{proof}

Finally, we will use the fact that if $V$ is a subvariety of $\mathbb P^n$ of dimension $\geq 2$  and the general hyperplane section of $V$ is ACM then $V$ itself must be ACM. A more general version can be found in \cite{HU}.

\begin{proposition}[\cite{mig-book} Theorem 1.3.3] \label{hyperplane sect acm}
Let $V$ be a locally Cohen-Macaulay, equidimensional closed subscheme of $\mathbb P^n$ and let $F$ be a general homogeneous polynomial of degree $d$ cutting out on $V$ a scheme $Z \subset V \subset \mathbb P^n$. Assume that $\dim V \geq 2$. Then $V$ is ACM if and only if $Z$ is ACM.
\end{proposition}


\section{Sufficient conditions for the top dimensional part and the radical 
to be ACM}

We begin with the case of $\mathbb P^3$. For convenience, in this case we refer to the above-mentioned schemes as $C$, $\overline C$ and $C^{top}$ to stress the fact that they are curves.

\begin{remark} \label{ci in linear case}

The following observations are elementary.

\begin{enumerate}

\item $C^{red}$ is  a union of lines.

\item A one-dimensional component of $C$ is non-reduced if and only if three or more hyperplanes in $\mathcal A$ contain the line defined by that component.

\item A non-reduced component $X$ of $\overline C$, viewed by itself, is a complete intersection. More precisely, if $e$ hyperplanes meet along the line $X^{red}$ then that component is a complete intersection of type $(e-1,e-1)$.

\end{enumerate}
\end{remark}

%

We first prove the main result mentioned in the introduction in the case of arrangements in $\mathbb P^3$.

\begin{theorem} \label{mainthm}
Let $\mathcal A$ be a plane arrangement in $\mathbb P^3$. 
Let $J$ be the Jacobian ideal associated to $\mathcal A$.
Assume that no plane of $\mathcal A$ is in the associated prime for more than one non-reduced  component of $\overline J$. Then  $R/{\overline J}$  is Cohen-Macaulay. That is,  $\overline C$ is  ACM.
\end{theorem}

\begin{proof}
Let $\overline{C} = X_1 \cup \dots \cup X_r \cup Y_1 \cup \dots \cup Y_s$, where $X_1,\dots,X_r$ are the non-reduced components and $Y_1 \cup \dots \cup Y_s$ are the reduced components. The hypothesis guarantees that if two or more of the $X_i$ have support lying on the same plane then that plane is not a hyperplane of $\mathcal A$.

Notice that there are exactly two planes in $\mathcal A$ through any of the $Y_i$. Furthermore, each $X_i$ and each $Y_i$ is ACM.

We first consider the non-reduced components. For each $X_i$ let $F_i$ be the product of the linear forms corresponding to hyperplanes in $\mathcal A$ containing the support of $X_i$. By Euler's theorem, $F_i \in I_{X_i}$. Our hypothesis guarantees that any two of the $F_i$ form a regular sequence. 

Consider the  ideal $(F_1, F_2)$. This defines a {\em reduced} union of lines, and our hypothesis guarantees that each of these lines is an entire component of $C$ (i.e. it is  one of the $Y_i$ and {\em not} the support of any of the $X_i$).

Let $D_{1,2}$ be the union of $X_1$, $X_2$ and the reduced lines defined by $(F_1,F_2)$. We now claim that  $D_{1,2}$ is ACM. Indeed, we have $F_1 \in I_{X_1}$, $F_2 \in I_{X_2}$, and $(F_1,F_2)$ is a regular sequence. Then liaison addition (Theorem \ref{LA})  guarantees that we have an equality of saturated ideals 
\[
I_{D_{1,2}} = F_1 \cdot I_{X_2} + F_2 \cdot I_{X_1}.
\]
Part (iv) of Theorem \ref{LA} further gives that 
 $D_{1,2}$ is ACM, since both $X_1$ and $X_2$ are.

Now notice that $F_1  F_2 \in I_{D_{1,2}}$. We will continue to add the $X_i$ one at a time until they are exhausted. As before, $(F_1F_2, F_3)$ is a regular sequence and $F_3 \in I_{X_3}$, so
\[
F_1F_2 \cdot I_{X_3} + F_3 I_{D_{1,2}}
\]
is the saturated ideal of the union of $X_1 \cup X_2 \cup X_3$ and a number of reduced components of $C$. Notice that no factor of $F_3$ contains any of the components of $D_{1,2}$. Denote this union by $D_{1,2,3}$. Liaison addition again gives that $D_{1,2,3}$ is ACM. Continuing inductively, we obtain an ACM curve $D_{1,\dots,r}$ consisting of the union of all the $X_i$ and some subset of the $Y_i$. 

Let $L$ be a linear form defining a hyperplane in $\mathcal A$ that is not a factor of any of the $F_i$. We claim that then $L$ does not vanish on any component of $D_{1,\dots,r}$. Indeed, by assumption $L$ does not contain the support of any of the $X_i$. Suppose that $L$ vanishes on one of the reduced lines in $D_{1,\dots,r}$. Such a line is already cut out by single factors of two of the $F_i$, so if $L$ also vanished on such a component then that component would be the support of one of the $X_i$; contradiction.

Now, the  saturated ideal
\[
L \cdot I_{D_{1,\dots,r}} + (F_1 \cdots F_r)
\]
defines the union of $D_{1,\dots,r}$ with the coplanar lines (on the plane defined by $L$) given by the complete intersection $(L,F_1\cdots F_r)$, which are all new components of $C$, and the union is ACM thanks to basic double linkage. Adding in all the remaining linear forms coming from $\mathcal A$ exhausts the reduced lines, preserving the ACM property, and we have shown that $\overline C$ is ACM.
\end{proof}

\begin{example} \label{emb pt}
Let $P \in \mathbb P^3$ be a point, let $L_1,L_2,L_3,L_4$ be four linear forms vanishing at $P$, and assume that any three of these linear forms share only the point $P$ in their common vanishing locus. Let $F = L_1L_2L_3L_4$ and let $J = (F_w,F_x,F_y,F_z)$ be the Jacobian ideal. Then $J$ is a saturated ideal with Betti diagram

\vfill \eject

\begin{verbatim}
        0    1    2    3
-------------------------
 0:     1    -    -    -
 1:     -    -    -    -
 2:     -    3    -    -
 3:     -    -    3    1
-------------------------
Tot:    1    3    3    1
\end{verbatim}   
(notice that $F_w, \dots,F_z$ are linearly dependent). The top dimensional part, $\overline C$, of course consists of six lines through $P$, which is easily seen to be ACM. But $J$ itself has an embedded point, so $R/J$ is not Cohen-Macaulay. The Hilbert polynomial of $R/J$ is $6t-1$ while that of the top dimensional part is $6t-2$.

So far our shared point $P$ comes because the arrangement is not essential. This can be viewed as taking 4 generic lines in $\mathbb P^2$, then coning, hence even though the ideal is saturated in $\mathbb P^3$, that embedded point is there because it was not saturated in $\mathbb P^2$. However, adding a general plane $L_5$ to the arrangement makes the arrangement essential while not changing anything locally at $P$. In fact the addition of $L_5$ provides a basic double link, and it is no surprise that now $R/J$ has Hilbert polynomial $10t - 9$ while $R/\overline J$ has Hilbert function $10t-10$. The embedded point is still there.
\end{example}

\begin{corollary}

If no three hyperplanes in $\mathcal A$ pass through the same line then $\overline C$ is ACM. Note that in this case the scheme $C$ defined by $J^{sat}$ may not be equidimensional (it may have embedded points -- see Example \ref{emb pt}) but the top dimensional part, $\overline C$, is reduced.

\end{corollary}

\noindent This corollary includes as a special case the star configurations. It also can be derived  from \cite{GHM3} Proposition 2.9 once one realizes that the assumption in \cite{GHM3} that the hyperplanes meet properly can be relaxed (in this case) to simply assume no three meet along a line.

Now we show that if we consider only the support of the singular locus, i.e. the radical $\sqrt{J}$, then under the same hypothesis as in Theorem \ref{mainthm} we always obtain an ACM union of lines.

\begin{corollary} \label{sing is acm}
Let $\mathcal A$ be a plane arrangement in $\mathbb P^3$. Assume that no hyperplane of $\mathcal A$ contains the support of more than one non-reduced component of $\overline C$.  Then the singular locus $C^{red}$ is ACM. That is, $R/\sqrt{J}$ is Cohen-Macaulay.

\end{corollary}

\begin{proof}
We have adopted  the assumption of Theorem \ref{mainthm}, i.e. we assume that no hyperplane of $\mathcal A$ contains the support of more than one non-reduced component of $\overline C$. In this case, in the proof of Theorem \ref{mainthm} we merely replace each $X_k$ by its support $X_i^{red}$ and carry out precisely the same construction.
\end{proof}

We now describe in  detail the third kind of ideal mentioned in the introduction, and we show how the same proof provides Cohen-Macaulayness for them. Let $\mathcal A$ be a plane arrangement in $\mathbb P^3$ defined by a product of linear forms $F$. Let $J$ be its Jacobian ideal and let $\sqrt{J} = \bigcap \mathfrak p_i$. For each $\mathfrak p_i$, let $L_{i,i},\dots,L_{i,a_i}$ be the factors of $F$ in $\mathfrak p_i$. Note that $\prod_{j=1}^{a_i} L_{i,j} \in \mathfrak p_i^{a_i}$ for all $i$, and in fact for each $i$ we have $\prod_{j=1}^{a_i} L_{i,j} \in \mathfrak p_i^{b_i}$ for any $0 \leq b_i \leq a_i$ (where we take the convention $\mathfrak p_i^0 = R$). Importantly, each $R/\mathfrak p_i^{b_i}$ is Cohen-Macaulay since $\mathfrak p_i$ is a complete intersection. 

Note also that  $a_i = 2$ for some $i$ if and only if the corresponding component of $\overline J$ is reduced, and in general we have $a_i \geq 2$ for all $i$. 
We now choose our integers $b_i$ as follows. If $a_i = 2$ then we take $b_i = 1$. If $a_i \geq 3$ then we choose any 
$0 \leq b_i \leq a_i$.
Then we have

\begin{corollary} \label{fat is acm}
Let $\mathcal A$ be a plane arrangement in $\mathbb P^3$ defined by a product, $F$, of linear forms. Let $J, \mathfrak p_i, a_i, b_i$ be as above. Assume that no factor of $F$ is in more than one $\mathfrak p_i$ for which the corresponding $a_i$ is  $\geq 3$. Then  $R/\bigcap \mathfrak p_i^{b_i}$ is Cohen-Macaulay.

\end{corollary}

\begin{proof}
Again the proof is the same as in Theorem \ref{mainthm}. This time we replace each height 2 primary ideal $\mathfrak q_i$ by $\mathfrak p_i^{b_i}$.
\end{proof}

With  the above results for $\mathbb P^3$ we can also prove the same  statements for arrangements in $\mathbb P^n$. To set up the notation, we now let $R = k[x_0,\dots,x_n]$ and we define $J$, $\overline J$, $\sqrt{J}$ and $\bigcap \mathfrak p_i^{b_i}$ in the same way as above, but now they represent codimension two subschemes of $\mathbb P^n$. In particular, $\overline J$  and $\bigcap \mathfrak p_i^{b_i}$ are unmixed and $\sqrt{J}$ is radical  (in particular, also unmixed). Let $V,$ $\overline{V}$ and $V^{red}$ be the schemes defined by  these three ideals.

\begin{corollary} \label{Pn}
Let $\mathcal A$ be a hyperplane arrangement in $\mathbb P^n$ for $n \geq 3$. Assume that no hyperplane of $\mathcal A$ contains the support of more than one non-reduced component of $\overline V$.  Then  $\overline V$,  $V^{red}$ and the scheme defined by $\bigcap \mathfrak p_i^{b_i}$ are ACM.
\end{corollary}

\begin{proof}

This is immediate using Proposition \ref{hyperplane sect acm}, by taking a general hyperplane section of~$\mathcal A$. \end{proof}

\begin{remark}
In the next section we will give examples of  plane arrangements in $\mathbb P^3$ for which $R/\sqrt{J}$ is not Cohen-Macaulay.
\end{remark}

We now give a dual version of Corollary \ref{Pn}.

\begin{corollary}
Let $Z$ be a set of points in $\mathbb P^n$. Let  $\ell_1$ and $\ell_2$ be distinct lines in $\mathbb P^n$. Let $Z_1 = \ell_1 \cap Z$ and $Z_2 = \ell_2 \cap Z$. 
Assume that one of the following holds for every choice of  $\ell_1$ and $\ell_2$:

\begin{itemize}

\item[(i)] $|Z_1| \leq 2$ or $|Z_2| \leq 2$;

\item[(ii)] $|Z_1| \geq 3$, $|Z_2| \geq 3$, and $\ell_1 \cap \ell_2 = \emptyset$;

\item[(iii)] If $|Z_1| \geq 3$, $|Z_2| \geq 3$ and $\ell_1 \cap \ell_2 = \{ P \} \neq \emptyset$ then $P \notin Z$;

\end{itemize}

\noindent Then for the dual arrangement $\mathcal A_Z$ we have $R/\sqrt{J}$ and $R/\overline{J}$ are Cohen-Macaulay.
\end{corollary}


\section{Illustrative Examples}

In this section we consider what can happen when the main assumption of Theorem \ref{mainthm} is not satisfied, and we study the Cohen-Macaulayness of our three different kinds of ideals (with emphasis on the first two kinds). To begin, we give an example, due to Mircea Musta\c t\v a and  the third author, which shows that the Jacobian ideal can be saturated and unmixed, but nevertheless $C = \overline C$ is not ACM.

\begin{example}[\cite{MS} Example 4.5] \label{hal ex}
Let $\mathcal A$ be the plane arrangement defined by the linear forms
\[
\begin{array}{ll}
x, \ y, \ z, \ w, \ x+y, \ x+z, \ x+w, \ y+z, \ y+w, \ z+w, \ x+y+z, \ x+y+w, \ x+z+w, \\
y+z+w, \ x+y+z+w.
\end{array}
\]
One can check that the singular locus is supported on 55 lines, of which 25 are the locus of intersection of three planes, while the remaining 30 are the intersection of two planes. Thus $\deg C = 25 \cdot 4 + 30 = 130$. However, there are many planes containing the support of two or more non-reduced components, so our theorem does not apply. And in fact, a computer check using CoCoA \cite{cocoa} give that $J$ is saturated and unmixed, but $C = \overline C$ is not ACM. Its Betti diagram is

\begin{verbatim}
        0    1    2    3
-------------------------
 0:     1    -    -    -
 1:     -    -    -    -
       ...
12:     -    -    -    -
13:     -    4    -    -
14:     -    -    -    -
15:     -    -    -    -
16:     -    -    -    -
17:     -    -    4    1
-------------------------
Tot:    1    4    4    1
\end{verbatim}

\noindent and its Hilbert polynomial is $130t - 1150$. However,  the radical $\sqrt{J}$ defines a scheme $C^{red}$ that {\em is} ACM. Its Betti diagram is

\begin{verbatim}
        0    1    2
--------------------
 0:     1    -    -
 1:     -    -    -
      ...
 8:     -    -    -
 9:     -   11   10
--------------------
Tot:    1   11   10
\end{verbatim}

\noindent  and the Hilbert polynomial is $55t - 275$.

\end{example}

\begin{example}
On the other hand, sometimes we get Cohen-Macaulayness even with a plane containing the support of three or more $X_i$. Let $\mathcal A$ be the plane arrangement in $\mathbb P^3$ defined by the linear forms 
\[
x, \ y, \ z, \ w, \ (x+y), \ (x+z), \ (x+w).
\]
Then the Jacobian ideal has Betti diagram 

\begin{verbatim}
        0    1    2
--------------------
 0:     1    -    -
 1:     -    -    -
        ...
 4:     -    -    -
 5:     -    4    -
 6:     -    -    3
--------------------
Tot:    1    4    3
\end{verbatim}

\noindent and Hilbert polynomial $24t - 64$  (in particular, $J = \overline{J}$). The radical has Betti diagram

\begin{verbatim}
        0    1    2
--------------------
 0:     1    -    -
 1:     -    -    -
 2:     -    -    -
 3:     -    -    -
 4:     -    6    5
--------------------
Tot:    1    6    5
\end{verbatim}
\noindent and Hilbert polynomial $15t-25$.

So  the lines $(x,y), (x,z), (x,w)$ all lie on the plane defined by $x$, and each of these lines has one more plane going through it, so those three lines are supports of non-reduced components of $\overline C$. But $R/J$ and $R/\sqrt{J}$ are Cohen-Macaulay (i.e. $C = \overline C$ and $C^{red}$ are both ACM).

\end{example}

\begin{example} \label{ex poss}
By looking at subsets of the planes in the Musta\c t\v a-Schenck example (Example \ref{hal ex}) we can produce arrangements to show that the ACMness of $\overline C$ and of $C^{red}$ are independent of each other.

\begin{itemize}

\item[(a)] {\em ($\overline C$ and $C^{red}$ are both ACM.)} Of course any arrangement where no three planes pass through the same line will give a star configuration, which is reduced and ACM. The smallest example  where $\overline C$ is not reduced is an arrangement of the form $xy(x+y)$. The smallest example where $\overline C$ is not reduced and the support is on at least two lines is an arrangement of the form $F = x y (x + y)(w + x + z)$. These follow easily from our methods.

\item[(b)] {\em ($C^{red}$ is ACM but $\overline C$ is not.)} As we have noticed, the Musta\c t\v a-Schenck example illustrates this possibility. 

\item[(c)] {\em ($\overline C$ is ACM but $C^{red}$ is not.)}  The smallest example that we found comes from an arrangement consisting of 8 planes, for instance the arrangement defined by 
\[
F = x y z w (x + y)( y + z)( z+w) ( w + x).
\]

\item[(d)] {\em (Neither is ACM.)} The smallest example that we found comes from an arrangement of 9 planes, for instance the arrangement defined by 
\[
F = x y z w (x + y)(y + z)(z+w)( w + x)( w + x + y + z).
\]

\end{itemize}

\end{example} 

\begin{example} \label{fat is CM}
The last example in Example \ref{ex poss} has interesting behavior for the third kind of ideal studied in this paper. Indeed, we have seen that $R/\sqrt{J}$ and $R/\overline J$ both fail to be Cohen-Macaulay for that example. Nevertheless, taking $b_i = 2$ for all  $i$ for which $a_i \geq 3$, the resulting algebra  is Cohen-Macaulay! In fact, taking each component of $\sqrt{J}$ to the power 2 (which is the symbolic square of $\sqrt{J}$ but does not arise in the construction of Corollary \ref{fat is acm}), i.e. considering $R/\bigcap \mathfrak p_i^2$, we still have Cohen-Macaulayness! It is interesting to have examples of reduced ideals that are not Cohen-Macaulay but whose symbolic square is Cohen-Macaulay.
\end{example}

\begin{example} \label{free not cm}

In this paper we have not been concerned with the question of whether the arrangement $\mathcal A$ is free or not. 
It is clear that if $\mathcal A$ is free then $R/\overline J$ is Cohen-Macaulay. But in fact there exist examples where $\mathcal A$ is free but $R/\sqrt{J}$ is not Cohen-Macaulay. For instance, taking 
$F = x y z w(x + y)( w + x)( y + z)( w + z)( w + y + z)( w + x + y + z) $ and $J$ the Jacobian ideal, we have that the Betti diagram for $R/\sqrt{J}$ is
\begin{verbatim}
        0    1    2    3
-------------------------
 0:     1    -    -    -
 1:     -    -    -    -
        ... 
 5:     -    -    -    -
 6:     -    9    9    1
-------------------------
Tot:    1    9    9    1
\end{verbatim}

\medskip

\noindent while the Betti diagram for $R/J$ is
\medskip

\begin{verbatim}
        0    1    2
--------------------
 0:     1    -    -
 1:     -    -    -
        ... 
 7:     -    -    -
 8:     -    4    -
 9:     -    -    -
10:     -    -    3
--------------------
Tot:    1    4    3
\end{verbatim}

\end{example}

\begin{example} \label{same comb}
Terao's conjecture posits that if two arrangements have the same combinatorics in their incidence then one is Cohen-Macaulay if and only if the other is. We do not know if this is true, but we remark that at least the resolutions of $J$, $\sqrt{J}$ and $\overline J$ can be different.    
Let 
\[
\begin{array}{c}
F = xyzw(x+y+z)(2x+y+z)(2x+3y+z)(2x+3y+4z)(3x+5z)(3x+4y+5z) \\
\hbox{and} \\
F' = xyzw(x+ y+ z)(2x+ y+ z)(2x+ 3y+ z)(2x+ 3y+ 4z)(x+ 3z)(x+ 2y+ 3z) 
\end{array}
\]
One can check on the computer that $F$ and $F'$ give arrangements with the same incidence lattice. Consider the minimal free resolutions of the various ideals considered in this paper.

\medskip

\begin{verbatim}
top dimensional part for F:        top dimensional part for F':

        0    1    2                        0    1    2
--------------------               --------------------
 0:     1    -    -                 0:     1    -    -
 1:     -    -    -                 1:     -    -    -
        ...                                ...
 7:     -    -    -                 7:     -    -    -
 8:     -    4    1                 8:     -    5    1
 9:     -    4    6                 9:     -    1    3
                                   10:     -    -    1
--------------------               --------------------
Tot:    1    8    7                Tot:    1    6    5
\end{verbatim}

\medskip

\begin{verbatim}
radical for F:                     radical for F':

        0    1    2                        0    1    2
--------------------               --------------------
 0:     1    -    -                 0:     1    -    -
 1:     -    -    -                 1:     -    -    -
        ...                               ...
 5:     -    -    -                 5:     -    -    -
 6:     -    4    -                 6:     -    4    1
 7:     -    -    3                 7:     -    1    3
 8:     -    1    1                 8:     -    1    1
--------------------               --------------------
Tot:    1    5    4                Tot:    1    6    5
\end{verbatim}

\medskip

\begin{verbatim}
Jacobian of F:                        Jacobian of F':

        0    1    2    3                     0    1    2    3
-------------------------            -------------------------
 0:     1    -    -    -              0:     1    -    -    -
 1:     -    -    -    -              1:     -    -    -    -
        ...                                  ...
 7:     -    -    -    -              7:     -    -    -    -
 8:     -    4    1    -              8:     -    4    1    -
 9:     -    -    -    -              9:     -    -    -    -
10:     -    -    -    -             10:     -    -    -    -
11:     -    -    -    -             11:     -    -    -    -
12:     -    -    -    -             12:     -    -    1    -
13:     -    -    6    4             13:     -    -    3    1
                                     14:     -    -    -    1
-------------------------            -------------------------
Tot:    1    4    7    4             Tot:    1    4    5    2
\end{verbatim}

\medskip

\begin{verbatim}

saturation for F:                     saturation for F':

The Jacobian is already saturated in both cases, so the Betti diagrams for 
the saturations are the same as above.
\end{verbatim}

\medskip

We also remark that the Hilbert polynomial for $R/J_F$ is $51t - 223$ while the Hilbert polynomial for $R/J_{F'}$ is $51t - 222$. This is at first somewhat surprising since the incidence lattice is the same in both cases, but the discrepancy comes because most factors of $F$ and $F'$ are only in three variables, and so this is reflecting a cone phenomenon related to the failure of saturation of the corresponding arrangement in $\mathbb P^2$ (which does not show up in the corresponding Hilbert polynomial  in $\mathbb P^2$ but does show up in the cone).

The construction used in Theorem \ref{mainthm} leads to the conclusion (using standard liaison arguments) that the  arrangements considered in that theorem {\em do} have resolutions that depend only on the combinatorics of the incidence lattice. Notice, though, that in this example the extra hypothesis of Theorem \ref{mainthm} is not satisfied. Indeed (for example), the line defined by $(x,z)$ and the line defined by $(x,x+y+z)$ both turn out to have three planes through them, and they share the plane defined by $x$.
\end{example}

\section{Graphic arrangements} \label{graphic}

We now give an application of the results of the previous sections. For a graph $G$ with vertices $v_i, i \in \{1,\ldots, n\} $ and edges $E= \overline{v_iv_j}$, the {\em graphic arrangement} $\mathcal A_G$ is defined by the hyperplanes $V(x_i-x_j)$ for every edge in $E$.
A result of Stanley  is that $A_G$ is a free arrangement if and only if $G$ is chordal, and an easy  localization argument extends this to show that if $G$ contains an induced $k-$cycle, then $S/J$ has projective dimension at least $k-1$  (see \cite{Kung-Schenck}  for both results).

\begin{corollary} \label{graphic result}

Let $G$ be a graph as above. Assume that no two 3-cycles of $G$ share an edge. For $\mathcal A_G$ denote by $J_G$ the corresponding Jacobian ideal. Then $R/\sqrt{J_G}$ and $R/{\overline J}_G$ are both Cohen-Macaulay. 

\end{corollary}

\begin{example}
If $G$ is any bipartite graph then $R/\sqrt{J_G}$ and $R/\overline{J}_G$ are Cohen-Macaulay.
\end{example}

\begin{example}
If $G$ is the 1-skeleton of the dodecahedron or of the  rhombicosidodecahedron

\medskip

\begin{center}

\includegraphics[width=4.6cm]{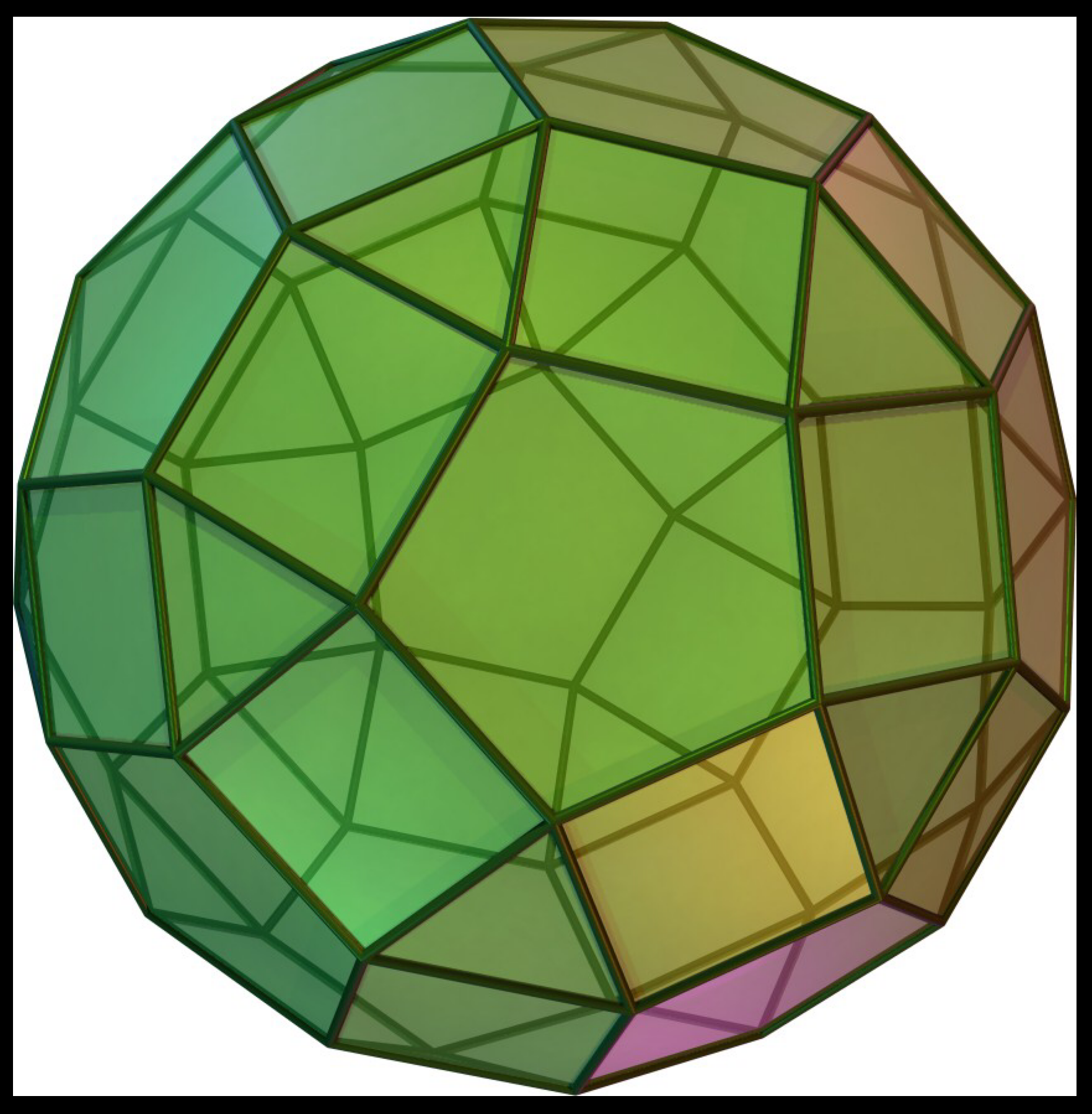}  

\end{center}

\medskip

\noindent then the corresponding algebras $R/\overline J$ and $R/\sqrt{J}$ are Cohen-Macaulay. (Picture from \linebreak wikipedia.org).
\end{example}

On the other hand, the next example shows that it is not true that either the radical or the top dimensional part is necessarily Cohen-Macaulay in general  for graphic arrangements. Thus it is good to have a sufficient condition.

\begin{example}
It is not the case that if we allow $G$ to have 3-cycles sharing an edge then $R/\sqrt{J_G}$ and $R/{\overline J}_G$ are necessarily Cohen-Macaulay. For example, take $G$ to be  the 1-skeleton of an octahedron (12 edges, 6 vertices). Then $R/\sqrt{J}$ has Betti diagram

\begin{verbatim}
        0    1    2    3
-------------------------
 0:     1    -    -    -
 1:     -    -    -    -
       ...
 8:     -    -    -    -
 9:     -   16   20    5
-------------------------
Tot:    1   16   20    5   
\end{verbatim}

\medskip

\noindent and Hilbert polynomial $ 50t - 230$,  while $R/\overline{J}$ has Betti diagram

\medskip

\begin{verbatim}
        0    1    2    3
-------------------------
 0:     1    -    -    -
 1:     -    -    -    -
       ...
 9:     -    -    -    -
10:     -    5    -    -
11:     -    1    2    -
12:     -    -    4    1
-------------------------
Tot:    1    6    6    1
\end{verbatim}

\medskip

\noindent and Hilbert polynomial $74t - 454$.
\end{example}


\section{Failure of Cohen-Macaulayness and liaison classes} \label{liaison sect}

We now give another application, which ties Jacobian ideals  of hyperplane arrangements in a more sophisticated way to liaison theory. In this section we will consider only equidimensional curves  in $\mathbb P^3$ with no embedded points. We recall that the {\em Hartshorne-Rao module} of a curve $C$ is defined by
\[
M(C) = \bigoplus_{t \in \mathbb Z} H^1(\mathbb P^3, \mathcal I_C(t)).
\]
This is a graded module over the polynomial ring $R$. The condition that $C$ is equidimensional with no embedded points implies (indeed, is equivalent to) the fact that $M(C)$ is of finite length.

Recall that two curves $C$ and $C'$ in $\mathbb P^3$ are in the same even liaison class if and only if $M(C)$ is isomorphic to some shift of $M(C')$ \cite{rao}. 
In particular, $M(C) = 0$ if and only if $C$ is ACM (recalling the classical fact that in codimension two, ACM is equivalent to being in the linkage class of a complete intersection).  Considering the exact sequence
\[
H^0(\mathcal O_{\mathbb P^3}(t)) \stackrel{\rho_t}{\longrightarrow} H^0(\mathcal O_C(t)) \rightarrow H^1(\mathcal I_C(t)) \rightarrow 0, 
\]
we see that being ACM in this setting is equivalent to the condition that the restriction map $\rho_t$ is surjective for all $t$.

This provides two natural numerical measures of the non-ACMness of $C$ as follows. Of course the simplest measure is simply $\dim_k M(C)$, ignoring the $R$-module structure of $M(C)$ completely and viewing it only as a $k$-vector space without grading. The best numerical measure of the failure of ACMness, though, is to  record, degree by degree, the failure of surjectivity of the above restrictions. That is, we can record the dimensions of the components of $M(C)$. Put differently, we can view $M(C)$ as a graded vector space and simply ignore the structure given by multiplication by elements of $R$. (Later in this section we will resume our consideration of the module structure and its connection to liaison.)

\begin{remark}
A celebrated theorem of Hartshorne and Hirschowitz \cite{HH} says that if $C$ is a general set of lines in projective space then the above restriction map on global sections always has maximal rank. Of course such a set of lines, even in $\mathbb P^3$, is never the singular locus of a hyperplane arrangement, but this result serves as a striking contrast to our situation.
\end{remark}

If $t_1$ is the smallest degree for which $\dim M(C)_t > 0$ and $t_n$ is the largest such degree, and if $a_i = \dim M(C)_i$ for $t_1 \leq i \leq t_n$, then we can consider the $n$-tuple of positive integers $(a_1,\dots,a_n)$ to measure, up to shift, the failure of ACMness, losing the shift but preserving the relative failure degree by degree. This gives the same measure of the failure of ACMness for all elements of the even liaison class of $C$. (This was known to Gaeta already in the 40's and 50's.) 

Proposition \ref{liaison addition for products} shows that in particular, if $(a_1,\dots,a_n)$ is the $n$-tuple measuring the failure of $\overline C_F$ (resp. $C_F^{red}$) to be ACM and $(b_1,\dots, b_q)$ is the $q$-tuple measuring the failure of $\overline C_G$ (resp. $C_G^{red}$) to be ACM then (assuming the hypotheses of the proposition) a suitable shifted sum of these tuples gives the failure of ACMness for the corresponding curve coming from $FG$.

It would be very nice to know which tuples $(a_1,\dots,a_n)$ arise as the failure of some $\overline C_F$ (resp. $C_F^{red}$) to be ACM. Even more ambitious would to to know which even liaison classes contain curves $\overline C_F$ (resp. $C_F^{red}$) coming from some plane arrangement in $\mathbb P^3$ (i.e. up to shift, which modules of finite length arise as the Hartshorne-Rao module of some $\overline C_F$ or $C_F^{red}$).

%

Toward this end, we show that for any integer $r \geq 0$ there is a curve $\overline C$ coming as the top dimensional part of some arrangement, where $\dim M(\overline C) = r$. We are only able to show this for the tuple $(r)$, i.e. for the case $n=1$. Notice that this uniquely determines the module structure of $M(\overline C)$. We denote the corresponding even liaison class by $\mathcal L_r$. We then give an analogous result for $C^{red}$.

\begin{theorem} \label{buchs thm}
Let $r \geq 1$ be a positive integer. Then:

\begin{itemize}

\item[\em (i)] There exists a positive integer $N$ and a product of linear forms $F$, defining an arrangement $\mathcal A_F$, such that 
\[
\dim M(\overline C_F)_{N} = r
\]
and all other components of $M(\overline C_F)$ are zero. Thus $\overline C_F$ fails by $r$ to be ACM, in the sense that the dimension of the cokernel of the restriction map is $r$ in one degree and 0 elsewhere.

\item[\em (ii)] For each $h \geq 1$ we can replace $N$ by $N+h$ and find a polynomial $G$ so that 
\[
 \dim M(\overline C_G)_{N+h} = r, 
 \]
and all other components of $M(\overline C_G)$ are zero.

\item[(iii)] The syzygy bundle for the Jacobian ideal of $F$ (or $G$) is locally free.
\end{itemize}
\end{theorem}

\begin{proof}
For the building block for our construction we could use the Musta\c t\v a-Schenck example (Example \ref{hal ex}), but in the interest of efficiency we take the subarrangement mentioned in Example \ref{ex poss} (d), namely 
$F = x y z w (x + y)(y + z)(z+w)( w + x)( w + x + y + z)$ which has 9 planes. For this we have the Betti diagram

\vfill \eject
\begin{verbatim}
        0    1    2    3
-------------------------
 0:     1    -    -    -
 1:     -    -    -    -
 2:     -    -    -    -
 3:     -    -    -    -
 4:     -    -    -    -
 5:     -    -    -    -
 6:     -    -    -    -
 7:     -    4    -    -
 8:     -    -    -    -
 9:     -    -    4    1
-------------------------
Tot:    1    4    4    1
\end{verbatim}

A theorem of Rao (\cite{rao} Theorem (2.5)) gives that if $Y$ is a curve in $\mathbb P^3$ with Hartshorne-Rao module $M(Y)$, and if the latter has minimal free resolution
\[
0 \rightarrow L_4 \stackrel{\sigma}{\longrightarrow} L_3 \rightarrow L_2 \rightarrow L_1 \rightarrow L_0 \rightarrow M(Y) \rightarrow 0,
\]
then $R/I_Y$ has a minimal free resolution of the form
\[
0 \rightarrow L_4 \stackrel{(\sigma_4,0)}{\longrightarrow} L_3 \oplus A \rightarrow B \rightarrow R \rightarrow R/I_Y \rightarrow 0
\]
where $A$ and $B$ are free modules. In our situation, since $M(\overline C)$ has finite length, and since the  resolution for $M(\overline C)$ is the dual of the resolution for $M(\overline C)^\vee$ (up to shift), we conclude that $M(\overline C)^\vee$ has minimal presentation
\[
R(-c-1)^4 \rightarrow R(-c) \rightarrow M(\overline C)^\vee \rightarrow 0
\]
for some integer $c$.
Hence $M(\overline C)^\vee \cong k \cong M(\overline C)$   (these isomorphisms are up to shift) is a one-dimensional $k$-vector space. Thus for our example, $\overline C$ is in the liaison class of two skew lines. 

Having this convenient building block, the rest follows from Proposition \ref{liaison addition for products} and Proposition \ref{BDL for arrangements}. If our $n$-tuple measuring the failure of ACMness is simply (1) as is the case for our building block, we can  take as our form $F$ the 9 forms in our example. This produces a curve $\overline C$ of degree 42, and one can check that $M(\overline C)$ occurs in degree 8.

Suppose we keep the tuple $(1)$ but we want an arrangement where the failure of ACMness comes in degree 9 instead of 8. By Proposition \ref{BDL for arrangements} we simply consider the arrangement associated to the form $LF$, where $L$ is a general linear form. The corresponding unmixed curve consists of the original $\overline C$ together with the complete intersection of $L$ and $F$. In other words, it is obtained from the original curve by basic double linkage using $F \in I_{\overline C}$ and a general $L \in [R]_1$. Hence the saturated ideal of this new curve is $L \cdot I_{\overline C} + F$ and the Hartshorne-Rao module is shifted by one, as desired. This can be done as often as we like, obtaining any rightward shift of $M(\overline C)$ as coming from a suitable planar arrangement.

Now we pass to dimension $r \geq 2$. To begin, we make $r$ copies of the building block by applying  general changes of coordinates. This guarantees that the corresponding arrangements, say $\mathcal A_{F_1}, \dots, \mathcal A_{F_r}$ will satisfy the hypotheses of Proposition \ref{liaison addition for products}. Note that $F_1,\dots,F_r$ all have the same degree, namely 9. 

\medskip

\underline{Claim}: {\em $\mathcal A_{F_1 \cdots F_r}$ produces the desired curve in $\mathcal L_r$.}

\medskip

To obtain the first of these, namely the tuple (2), we simply consider the curve defined by the saturated ideal
\[
F_2 \cdot I_{\overline C_{F_1}} + F_1 \cdot I_{\overline C_{F_2}}.
\]
(The only non-reduced components come from factors of $F_1$ and  from factors of $F_2$ so they are already accounted for in $\overline C_{F_1}$ and $\overline C_{F_2}$, and the only new components come from the complete intersection of $F_1$ and $F_2$.) This comes from the arrangement $\mathcal A_{F_1 F_2}$. 
The resulting curve $\overline C_{F_1 F_2}$ has Hartshorne-Rao module of dimension 2  in degree $8 + 9 = 17$. The degree of this curve is $42 + 42 + (9)(9) = 165$.

For the next curve we form the saturated ideal
\[
F_3 \cdot I_{\overline C_{F_1 F_2}} + F_1 F_2 \cdot I_{\overline C_{F_3}}.
\]
By the same reasoning, the resulting curve is $\overline C_{F_1 F_2 F_3}$. Its Hartshorne-Rao module has dimension 3, occurring in degree $17 + 9 = 8 + 18 = 26$. Proceeding inductively gives the claim. This completes the proof of (i).

For (ii), we can produce any rightward shift of curves we have constructed by applying Proposition \ref{BDL for arrangements}. 

For (iii) one can check on the computer that the Jacobian ideal is not saturated, but its saturation is unmixed (but not Cohen-Macaulay), so its syzygy bundle is locally free. Then apply Corollary \ref{loc free}.
\end{proof}

We summarize the above result more compactly and observe a consequence.

\begin{corollary} \label{buchs even liaison}
Let $r \geq 1$ be a positive integer. Let $\mathcal L_r$ be the (Buchsbaum) even liaison class whose associated Hartshorne-Rao module (up to shift) has dimension $r$ supported in only one degree. Then $\mathcal L_r$  contains infinitely many curves (of infinitely many degrees, lying in infinitely many shifts of $\mathcal L_r$) arising as the top dimensional part of the unmixed scheme coming from some arrangement in $\mathbb P^3$.
\end{corollary}

\begin{proof}
This follows immediately from Theorem \ref{buchs thm}. 
\end{proof}

\begin{remark}
If we had used the Musta\c t\v a-Schenck example (Example \ref{hal ex}) for our building block in Theorem \ref{buchs thm}, we would be able to add a second part to the above corollary saying that the radicals of our constructed ideals are Cohen-Macaulay. However, we preferred to use the smallest degrees possible.
\end{remark}

We now show that the same approach gives an analogous result for the radicals of Jacobian ideals.

\begin{corollary} \label{main thm for rad}
Let $r \geq 1$ be a positive integer. Let $\mathcal L_r$ be the (Buchsbaum) even liaison class whose associated Hartshorne-Rao module (up to shift) has dimension $r$ supported in only one degree. Then $\mathcal L_r$  contains infinitely many reduced curves (of infinitely many degrees, lying in infinitely many shifts of $\mathcal L_r$) whose homogeneous ideal is the radical of the Jacobian ideal of some arrangement in $\mathbb P^3$.
\end{corollary}

\begin{proof}
The proof is identical, except that we need a suitable building block. One can check that there are arrangements of 8 planes that will do the trick. For example, 
\[
F = yz(x + y)( x + z)( w + x)( x + y + z)( w + x + y)( w + x + z)
\]
gives a Jacobian ideal whose radical has Betti diagram
\begin{verbatim}
        0    1    2    3
-------------------------
 0:     1    -    -    -
 1:     -    -    -    -
 2:     -    -    -    -
 3:     -    -    -    -
 4:     -    -    -    -
 5:     -    8    8    1
-------------------------
Tot:    1    8    8    1
\end{verbatim}
\noindent and this clearly has Hartshorne-Rao module of dimension 1, supported in degree 4. 
\end{proof}

\begin{example} \label{not all classes}
It is also natural to ask how many (all?) even liaison classes contain curves of the form $\overline C$ or $C^{red}$. In any case it is not true that $M(\overline C)$ or $M(C^{red})$ is necessarily supported in only one degree. For example, for 

\[
F = {\small 
\begin{array}{l}
w x y z (w + 3x + 5y + 7z)( w + x)( w + y)( w + z)( 2w + 3x + 5y + 7z)( x + y)( x + z) \\
( w + 4x + 5y + 7z)( y + z)( w + 3x + 6y + 7z)( w + 3x + 5y + 8z)( w + x + y)( w + x + z) \\
( 2w + 4x + 5y + 7z)( w + y + z)( 2w + 3x + 6y + 7z)( 2w + 3x + 5y + 8z)( x + y + z) \\
( w + 4x + 6y + 7z)( w + 4x + 5y + 8z)( w + 3x + 6y + 8z)( w + x + y + z)( 2w + 4x + 6y + 7z) \\
( 2w + 4x + 5y + 8z)( 2w + 3x + 6y + 8z)( w + 4x + 6y + 8z)( 2w + 4x + 6y + 8z)
\end{array}   }
\]

\noindent the Betti diagram of $R/\overline J$ is
\begin{verbatim}
        0    1    2    3
-------------------------
 0:     1    -    -    -
 1:     -    -    -    -
         ...
28:     -    -    -    -
29:     -    5    -    -
30:     -    -    -    -
         ...
34:     -    -    -    -
35:     -    5   10    -
36:     -    -    -    -
37:     -    -    -    1
-------------------------
Tot:    1   10   10    1
\end{verbatim}

\medskip

\noindent Since the last matrix in the minimal free resolution is a $(1 \times 10)$ matrix of cubics, the Hartshorne-Rao module cannot possibly be supported in one degree.

And for the same arrangement, the Betti diagram of $R/\sqrt{J}$ is

\vfill \eject

\begin{verbatim}
        0    1    2    3
-------------------------
 0:     1    -    -    -
 1:     -    -    -    -
        ...
22:     -    -    -    -
23:     -   28   27    5
24:     -    9   26   12
-------------------------
Tot:    1   37   53   17
\end{verbatim}


\medskip

\noindent Thus the dual module $M(C^{red})^\vee$ is generated in more than one degree, so $M(C^{red})$ cannot be supported in one degree.
\end{example}

\begin{remark}
In the construction from Theorem \ref{buchs thm}, to produce the curve in $\mathcal L_2$ we required $9+9=18$ planes. We remark that there exists a set of 11 planes for which the top dimensional part of the Jacobian ideal lies in this liaison class, namely that coming from $F = xy z w (x + y)( x + z)( w + x)( y + z)( w + y)( w + z)( w + x + y + z)$. Its Betti diagram is

\begin{verbatim}
        0    1    2    3
-------------------------
 0:     1    -    -    -
 1:     -    -    -    -
         ...
 8:     -    -    -    -
 9:     -    4    -    -
10:     -    -    -    -
11:     -    3    8    2
-------------------------
Tot:    1    7    8    2
\end{verbatim}

\medskip

\noindent One can check  on the computer that the corresponding Hartshorne-Rao module is, indeed, 2-dimensional, only in degree 10, by looking at the last map in the resolution.
It follows from this, using our methods,  that there is an arrangement of $11+9 = 20$ planes whose corresponding Hartshorne-Rao module is 3-dimensional, and an arrangement of $11+11 = 22$ planes whose module is 4-dimensional. 
Turning to the radical, we can make a similar analysis. The point, though, is that 
 our construction is  not optimal.

\end{remark}


\section{Open questions} \label{open q}

For simplicity we focus on the case of planar arrangements in $\mathbb P^3$.

\begin{enumerate}

\item Corollary \ref{buchs even liaison} shows that every (Buchsbaum) even liaison class whose associated \linebreak Hartshorne-Rao module is supported in one degree contains curves arising as the top dimensional part of the Jacobian ideal of some planar arrangement, as well as curves arising as the radical of the Jacobian ideal of some planar arrangement. Example \ref{not all classes} shows that there are other even liaison classes containing such curves. Can we classify the even liaison classes 
 (in terms of the Hartshorne-Rao module) that contain such curves? Is it possible that all even liaison classes do? A possible test case for a negative answer is the Hartshorne-Rao module that is one-dimensional in each of two consecutive degrees, with the multiplication $\times L$ from the first component to the second always zero. More generally, an interesting special case is the following:

\medskip

\begin{quotation}
{\em 
If $\mathcal L$ is an even Buchsbaum liaison class of curves in $\mathbb P^3$ containing curves that arise as the top dimensional part of the Jacobian ideal of some hyperplane arrangement, must the corresponding Hartshorne-Rao module be supported in one degree?}
\end{quotation}

\medskip

\noindent We can ask the same questions for the radical of the Jacobian ideals.

\item Let $\mathcal L$ be an even liaison class that contains curves arising either as the top dimensional part of the Jacobian ideal, or as the radical of the Jacobian ideal, of some planar arrangement in $\mathbb P^3$. Is there any structure for the collection of such curves, analogous in some way to the Lazarsfeld-Rao property \cite{BM2}, \cite{BBM}?


\item  Suppose that some planes of $\mathcal A$ contain the support of two or more non-reduced components of $\overline C$. Under what conditions does $\overline C$ still have to be ACM? 

\item Terao's conjecture says that the question of whether $\mathcal A_F$ is free (i.e. whether $R/J_F$ is Cohen-Macaulay) depends only on the combinatorics of the incidence lattice $L(\mathcal A_F)$. Is it similarly true that the question of whether $R/\sqrt{J_F}$ and/or $R/\overline J_F$ are Cohen-Macaulay depends only on the combinatorics of $L(\mathcal A_F)$? Notice that Example \ref{same comb} shows that in any case the Betti diagram can change even among arrangements with the same lattice, but the question of Cohen-Macaulayness is not yet so clear.

\item Can we extend these results to $\mathbb P^n$? In this case there is a collection of Hartshorne-Rao modules, $M^i(X) = \bigoplus_{t \in \mathbb Z} H^i(\mathcal I_X(t))$, $1 \leq i \leq \dim X$. How do the higher values of $i$ play a role?  Is it more efficient to consider the stable equivalence classes of the sheafification of the syzygy modules instead of considering $M^i(X)$?

\end{enumerate}



\begin{thebibliography}{BDHHSS}
\bibitem{cocoa} J.~Abbott, A.M.~Bigatti, and L.~ Robbiano, 
  CoCoA : a system for doing Computations in Commutative Algebra, 
  available at {\tt http://cocoa.dima.unige.it}.
  
  \bibitem{BBM} E. Ballico, G. Bolondi and J. Migliore, {\em The Lazarsfeld-Rao problem for liaison classes of two-codimensional subschemes of $\mathbb P^n$}, Amer. J. Math. {\bf 113} (1991), 117--128.
  
  \bibitem{BM1} G. Bolondi and J. Migliore, {\em Classification of maximal rank curves in the liaison class $\bf L_n$}, Math. Ann. {\bf 277} (1987), 585--603.
  
  \bibitem{BM2} G. Bolondi and J. Migliore, {\em The structure of an even liaison class}, Trans. Amer. Math. Soc. {\bf 316} (1989), 1--37.
  
\bibitem{DS} A. Dimca and G. Sticlaru, {\em Saturation of Jacobian ideals: Some applications to nearly free curves, line arrangements and rational cuspidal plane curves}, J. Pure Appl. Algebra {\bf 223} (2019), 5055--5066.

\bibitem{GHM3} A.V. Geramita, B. Harbourne and J. Migliore, {\em Star configurations in $\mathbb P^3$}, J. Algebra {\bf 376} (2013), 279--299.

\bibitem{GM4} A.V. Geramita and J. Migliore, {\em A generalized liaison addition}, J. Algebra {\bf 163} (1994), 139--164.

\bibitem{HH} R. Hartshorne and A. Hirschowitz, {\em Droites en position g\'en\'erale dans l'espace projectif}, Algebraic geometry (La R\'abida, 1981), 169--188, Lecture Notes in Math., 961, Springer, Berlin, 1982. 

\bibitem{HU} C.\ Huneke and B.\ Ulrich,
{\em General Hyperplane Sections of Algebraic Varieties},
J.\ Alg.\ Geom.\  {\bf 2} (1993), 487--505.

\bibitem{Kung-Schenck}
J. Kung and H. Schenck, {\em Derivation modules of orthogonal duals of hyperplane arrangements}, 
J. Algebraic Combin. {\bf 24} (2006), no. 3, 253--262. 

\bibitem{LR} R. Lazarsfeld and A.P. Rao, {\em Linkage of general curves of large degree}," Algebraic geometry -- open problems (Ravello, 1982), 267--289, 
Lecture Notes in Math., {\bf 997}, Springer, Berlin, 1983. 

\bibitem{mig-book} J.\ Migliore, ``Introduction to Liaison Theory and
Deficiency Modules,''  Birkh\"auser, Progress in Mathematics {\bf 165}, 1998.

\bibitem{MN3} J. Migliore and U. Nagel, {\em Reduced arithmetically Gorenstein schemes and simplicial polytopes with maximal Betti numbers}, Adv. Math. {\bf 180} (2003), 1--63.

\bibitem{MS} M. Musta\c t\v a and H. Schenck, {\em The module of logarithmic $p$-forms of a locally free arrangement}, J. Algebra {\bf 241} (2001), 699--719.

\bibitem{OT}  P. Orlik and H. Terao, ``Arrangements of hyperplanes," Grundlehren der Mathematischen Wissenschaften, {\bf 300}. Springer-Verlag, Berlin, 1992. 

\bibitem{rao}
A.P. Rao, {\em  Liaison among curves in $\mathbb P^3$}, Invent. Math. {\bf 50} (1978/79), no. 3, 205--217.

\bibitem{schenck}
H. Schenck, {\em Elementary modifications and line arrangements in $\mathbb P^2$}, Comment. Math. Helv. {\bf 78} (2003), 447--462.

\bibitem{Sw} P.\ Schwartau, `` Liaison Addition and Monomial Ideals,'' 
Ph.D.\ thesis, Brandeis University (1982) (thesis advisor: D.\ Eisenbud).

\end{thebibliography}
\end{document}